\documentclass[11pt]{amsart}

\usepackage{latexsym}
\usepackage{color}

\usepackage{tikz-cd}
\usepackage[all]{xy}
 \usepackage[latin1]{inputenc} 
\usepackage[english]{babel}
\usepackage{mathtools}
\usepackage[T1]{fontenc}
 \usepackage{lmodern}
\usepackage{mathrsfs}
\usepackage{enumitem}

\usepackage{graphicx,epsfig,amscd,graphics,xypic,subcaption}
 \usepackage[all]{xy}
\usepackage{amsmath,amssymb}

\usepackage[margin=4cm]{geometry}

\usepackage{hyperref}
\hypersetup{
    bookmarks=true,         
    unicode=false,          
    pdftoolbar=true,        
    pdfmenubar=true,        
    pdffitwindow=false,     
    pdfstartview={FitH},    
    pdftitle={Affine surfaces and their Veech groups},    
    pdfauthor={Eduard Duryev, Charles Fougeron, and Selim Ghazouani},     
    pdfsubject={Subject},   
    pdfcreator={Creator},   
    pdfproducer={Producer}, 
    pdfkeywords={Veech group} {Affine surface} {Affine structure}, 
    pdfnewwindow=true,      
    colorlinks=false,       
    linkcolor=red,          
    citecolor=green,        
    filecolor=magenta,      
    urlcolor=cyan           
}

\newcommand{\FVS}{\mathbf{V}\left({\Sigma}\right)}
\newcommand{\FV}{\mathbf{V}}
\newcommand{\FVSo}{\mathbf{V}_0\left({\Sigma}\right)}

\newcommand{\sk}{\smallskip}

\newcommand{\Aff}{\mathrm{Aff}(\mathbb{C})}
\newcommand{\Afr}{\mathrm{Aff}_{\mathbb{R^*_+}}(\mathbb{C})}
\newcommand{\C}{\mathbb{C}}
\newcommand{\Z}{\mathbb{Z}}
\newcommand{\R}{\mathbb{R}}

\newcommand{\Hy}{\mathbb{H}}

\newcommand{\SL}{\mathrm{SL}}
\newcommand{\SLtR}{\mathrm{SL}_2 \! \left( \R \right)}
\newcommand{\D}{\mathbb{D}}
\newcommand{\pD}{\partial \mathbb{D}}

\newcommand{\DD}{\operatorname{D}}

\newtheorem{theorem}{Theorem}
\newtheorem*{theorem*}{Theorem}

\newtheorem{lemma}{Lemma}
\newtheorem{proposition}{Proposition}

\newtheorem*{remark}{Remark} 
\theoremstyle{definition}
\newtheorem{definition}{Definition}

\begin{document}

\newcommand{\GMP}{two-chamber}
\newcommand{\DbDu}{disco }

\title{Affine surfaces and their Veech groups.}
\author{Eduard Duryev}
\author{Charles Fougeron}
\author{Selim Ghazouani}
\maketitle

\begin{center}

{\it To the memory of William Veech}

\end{center}

\begin{abstract}
We introduce a class of objects which we call \textit{'affine surfaces'}. 
These provide families of foliations on surfaces whose dynamics we are interested in. 
We present and analyze a couple of examples, and we define concepts related to these 
in order to motivate several questions and open problems. In particular we generalise the notion of Veech group to affine surfaces, and we prove a structure result about these Veech groups.
\end{abstract}

\section*{Notations}

\begin{itemize}
\item $\Sigma_g$ or only $\Sigma$ is a compact surface of genus $g \geq 2$;
\item $\mathrm{Gl}_2^+(\R)$ is the group of 2 by 2 matrices whose determinant is positive;
\item $\SLtR$ is the group of 2 by 2 matrices whose determinant is equal to $1$;
\item $\Aff$ is the one dimensional affine complex group, $\C^* \ltimes \C$;
\item $\Afr$ is the subgroup of $\Aff$ of elements whose linear parts are real positive;
\item $\Hy$ is the upper half plane of $\C$;

\end{itemize}

\section{Introduction.}

A \textit{translation} structure on a surface is a geometric structure
modelled on the complex plane $\C$ with structural group the set of
translations. A large part of the interest that these structures have
drawn lies in the directional foliations inherited from the standard
directional foliations of $\C$ (the latter being invariant under the
action of translations). Examples of such structures are polygons
whose sides are glued along parallel sides \textit{of same length},
see Figure \ref{translation} below.

\begin{figure}[!h]
  \centering
  \includegraphics[scale=0.3]{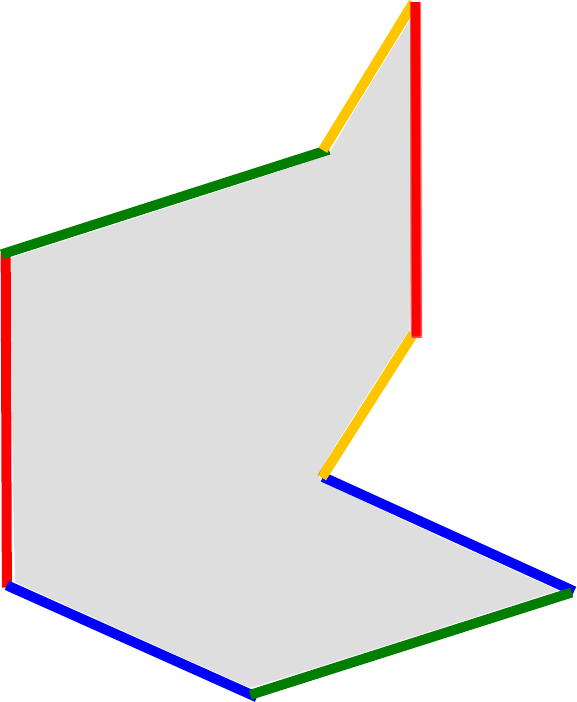}
  \caption{A translation surface of genus $2$.}
  \label{translation}
\end{figure}

The directional foliation, say in the horizontal direction, can be
drawn very explicitly: a leaf is a horizontal line until it meets a
side, and continues as the horizontal line starting from the point on
the other side to which it is identified.  These foliations have been
more than extensively studied over the past forty years.  They are
very closely linked to one dimensional dynamical systems called
\textit{interval exchange transformations}, and most of the basic
features of these objects (as well as less basic ones !) have long
been well understood, see \cite{Zorich} for a broad and clear survey
on the subject.

\vspace{2mm} The starting point of this article is the following
remark: to define the horizontal foliation discussed in the example
above, there is no need to ask for the sides glued together to have
same length, but only their being parallel in which case we can glue
along affine identifications. In terms of geometric structures, it
means that we extend the structural group to all the transformation of
the form $z \mapsto az +b$ with $a \in \R^*_+$ and $b \in
\C$. Formally, these corresponds to (branched) complex affine
structures whose holonomy group lies in the subgroup of $\Aff$ whose
linear parts are real positive. A simple example of such an
\textit{'affine surface'} is given by the gluing below:

\begin{figure}[!h]
  \centering
  \includegraphics[scale=0.3]{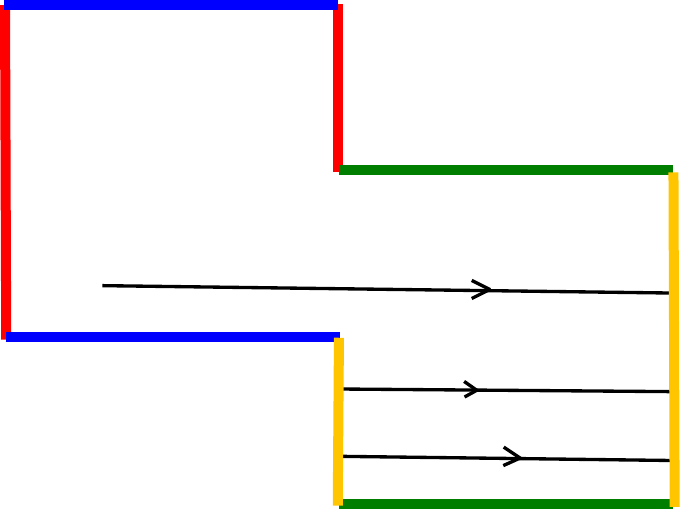}
  \caption{An \textit{'affine surface'} of genus $2$ and a leaf of its
    horizontal foliation.}
  \label{example}
\end{figure}

A notable feature of these affine surfaces is that they present
dynamical behaviours of \textit{hyperbolic} type: the directional
foliations sometimes have a closed leaf which 'attracts' all the
nearby leaves. This is the case of the closed leave drawn in black on
Figure \ref{example2} below. This situation is in sharp contrast with
the case of translation surfaces and promise a very different picture
in the affine case.

\begin{figure}[!h]
  \centering
  \includegraphics[scale=0.3]{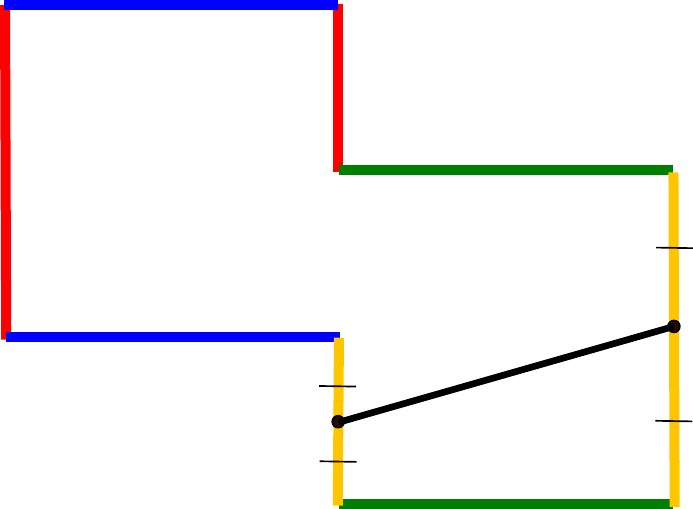}
  \caption{A \textit{'hyperbolic'} closed leaf.}
  \label{example2}
\end{figure}

It is somewhat surprising to find no systematic study of these
\textit{'affine surfaces'} in the literature. However, related objects
and concepts have kept poping up every now and then, both of geometric
and dynamical nature. To our knowledge, their earliest appearance is
in the work of Prym on holomorphic $1$-forms with values in a flat
bundle, see \cite{Prym}. These provide an algebraico-geometric
interpretation of these affine surfaces, see also
Mandelbaum(\cite{Mandelbaum, Mandelbaum2}) and Gunning
(\cite{Gunning}). We would also like to mention Veech's remarkable
papers \cite{Veech} and \cite{Veech1} where he investigates moduli
spaces of complex affine surfaces with singularities as well as
Delaunay partitions for such surfaces.  On the dynamical side, the
first reference to related questions can be found in Levitt's paper on
foliations on surfaces, \cite{Levitt2} where he builds an affine
interval exchange (AIET) with a wandering interval (these AIETs must
be thought of as the one-dimensional reduction of the foliations we
are going to consider). It is followed by a serie of works initiated
by Camelier and Gutierrez \cite{CamelierGutierrez} and pursued by
Bressaud, Hubert and Maass \cite{BressaudHubertMaass}, and Marmi,
Moussa, and Yoccoz \cite{MarmiMoussaYoccoz}. They generalize a well
known construction of Denjoy to build out of a standard IET an AIET
having a wandering interval, behaviour which is (conjecturally) highly
non-generic. Very striking is that the question of the behavior of a
typical AIET has been very little investigated. In this direction, we
mention the nice article of Liousse \cite{Liousse} where the author
deals with the topological generic behaviour of transversely affine
foliations on surfaces.

\subsection*{Contents of the paper and results}

After introducing formal definitions as well as a couple of
interesting examples of \textit{'affine surfaces'}, we prove a
structure result about Veech groups of affine surfaces.

The Veech group of an affine surface $\Sigma$ is the straightforward
generalization of its translation analogue: it is the subgroup of
$\SLtR$ made of linear parts of locally affine transformation of
$\Sigma$. It is a well-known fact that the Veech group of a
translation surface is always discrete. This fails to be true in the
more general case of affine surface, although the examples of surfaces
whose Veech group is not discrete are fairly distinguishable. We
completely describe the class of surfaces whose Veech group is not
discrete. Roughly, those are the surface obtained starting from a
ribbon graph and gluing to its edges a finite number of
\textit{'affine cylinders'}. We call such surfaces \textbf{Hopf
  surfaces}, because they must be thought of as higher genus analogues
of Hopf tori, that are quotients $\C^* /_{(z \sim \lambda z)}$ with
$\lambda$ a positive constant different from $1$. Precisely we prove

\newtheorem*{thm:classification}{Theorem \ref{classification}}
\begin{thm:classification}
  Let $\Sigma$ be an affine surface of genus $\geq 2$.
  There are two possible cases :

  \begin{enumerate}

  \item $\FVS$ is the subgroup of upper triangular elements of $\SLtR$ and $\Sigma$ is a Hopf surface.

  \item $\FVS$ is discrete.

  \end{enumerate}
\end{thm:classification}

We also prove the following theorem on the existence of closed
geodesics in genus $2$ :

\newtheorem*{thm:cylinderingenus2}{Theorem \ref{cylinderingenus2}}
\begin{thm:cylinderingenus2}
  Any affine surface of genus $2$ has a closed regular geodesic.
\end{thm:cylinderingenus2}

The proof is elementary and relies on combinatoric
arguments. Nonetheless it is a good motivation for a list of open
problems we address in Section \ref{open}. We end the article with a
short appendix reviewing Veech's results on the geometry of affine
surfaces contained in the article \cite{Veech1} and the unpublished
material \cite{VeechU} that W. Veech kindly shared.

\subsection*{About Bill Veech's contribution.}

Bill Veech's sudden passing away encouraged us to account for his
important contribution to the genesis of the present article. About
twenty years ago, he published a very nice paper called {\it Delaunay
  partitions } in the journal \textit{Topology} (see \cite{Veech1}),
in which he investigated the geometry of complex affine surfaces (of
which our 'affine surfaces' are particular cases). A remarkable result
contained in it is that affine surfaces all have geodesic
triangulations in the same way flat surfaces have. We used it
extensively when we first started working on affine surfaces,
overlooking the details of \cite{Veech1}. But at some point, we
discovered a familly of affine surfaces that seemed to be a
counter-example to Veech's result and which provides an obstruction
for affine surfaces to have a geodesic triangulation. We then decided
to contact Bill Veech, who replied to us almost instantly with the
most certain kindness. He told us that he realized the existence of
the mistake long ago, but since the journal \textit{Topology} no
longer existed and that the paper did not draw a lot of attention, he
did not bother to write an erratum. However, he shared with us courses
notes from 2008 in which he \textit{'fixed the mistake'}. It was a
pleasure for us to discover that in these long notes (more than 100p)
he completely characterizes the obstruction for the slightly flawed
theorem of \textit{Delaunay partitions} to be valid, overcoming
serious technical difficulties. We extracted from the notes the
Proposition \ref{closedgeodesic} which is somewhat the technical
cornerstone of this paper.

A few weeks before his passing away, Bill Veech allowed us to
reproduce some of the content of his notes in an appendix to this
article. It is a pity he did not live to give his opinion and modify
accordingly to his wishes this part of the paper.

\subsection*{Acknowledgements.}

We are very grateful to Vincent Delecroix, Bertrand Deroin, Pascal
Hubert, Erwan Lanneau, Leonid Monid and William Veech for interesting
discussions. The third author is grateful to Luc Pirio for introducing
him to the paper \cite{Veech1}. The third author acknowledges partial
support of ANR Lambda (ANR-13-BS01-0002).

\section{Affine surfaces.}
\label{basics}

We give in this section formal definitions of \textit{affine surfaces}
and several concepts linked to both their geometry and dynamics.

\subsection{Basics.}

An affine structure on a complex manifold $M$ of dimension $n$ is an
atlas of chart $(U_i, \varphi)$ with values in $\C^n$ such that the
transition maps belong to
$\mathrm{Aff}_n(\C) = \mathrm{Gl}_n(\C) \ltimes \C^n$. It is well
known that the only compact surfaces (thought of as $2$-dimensional
real manifolds) carrying an affine structure are tori. We make the
definition of an affine structure less rigid, by allowing a finite
number of points where the structure is singular, in order to include
the interesting examples mentioned in the introduction:

\begin{definition}

  \begin{enumerate}
  \item An \textbf{affine surface} $\mathcal{A}$ on $\Sigma$ is a
    finite set $\mathcal{S} = \{ s_1, \ldots, s_n \} \subset \Sigma$
    together with an affine structure on
    $\Sigma \setminus \mathcal{S}$ such that the latter extends to a
    euclidean cone structure of angle a multiple of $2\pi$ at the
    $s_i$'s.

  \item A \textbf{real affine surface} is an affine surface whose
    structural group has been restricted to $\Afr$.
  \end{enumerate}
\end{definition}

The type of singularities we allow results of what seems to be an
arbitrary choice. We could have as well allowed singular points to
look like affine cones, or the angles to be arbitrary. We will justify
our choice very soon. A first important remark is that an affine
surface satisfy a \textit{discrete Gauss-Bonnet} equality. If
$\mathcal{S} = \{ s_1, \ldots, s_n \}$ is the set of singular points
of an affine structure on $\Sigma$; and $k_i \geq 2$ is the integer
such that the cone angle at $s_i$ is $2k_i\pi$, then

$$ \sum_{i=1}^n{1-k_i} = \chi(\Sigma) = 2-2g $$ 

From now on and until the end of the paper, we will consider only
\textit{real affine surfaces}, and therefore we will refer to those
only as \textit{affine surface}.

\vspace{2mm} A general principle with geometric structures is that any
object that is defined on the model and is invariant under the
transformation group is well defined on the manifolds carrying such a
structure. In our case the model is $\C$ with structure group
$\Afr$. Among others, angles and (straight) lines are well defined on
affine surfaces. More striking is the fact that it makes sense to say
that the orientation of a line is well defined, and for each angle
$\theta \in S^1$ we can define a \textit{foliation} oriented by
$\theta$ that we denote by $\mathcal{F}_{\theta}$, whose leaves are
exactly the lines oriented by $\theta$.  \\ Finally remark that
although the speed of a path is only defined up to a fixed constant,
it makes sense to say that a path has constant speed (speed which is
not itself well defined), as well as to say that a path has finite or
infinite length.

\vspace{2mm} Formally:

\begin{itemize}

\item a \textbf{geodesic} is an affine immersion of a segment $]a,b[$
  ($a$ or $b$ can be $\pm \infty$);

\item a \textbf{saddle connection} is a geodesic joining two singular
  points;

\item a \textbf{leaf} of a directional foliation is a maximal geodesic
  in the direction of the foliation;

\item a \textbf{closed geodesic} (or \textbf{closed leaf} if the
  direction of the foliation is unambiguous) is an affine embedding of
  $\mathbb{R}/\Z$;

\item the first return on a little segment orthogonal to such a closed
  geodesic is a map of the form $x \mapsto \lambda x$ with
  $\lambda \in \R^*_+$. We say it is \textbf{flat} if $\lambda = 1$
  and that it is \textbf{hyperbolic} otherwise.

\end{itemize}

\subsection{Cylinders.}
\label{cylinders}

These definitions being set, we introduce a first fundamental example,
the \textit{Hopf torus}.  Consider a real number $\lambda \neq 1$ and
identify every two points on $\mathbb C^*$ which differ by scalar
multiplication by $\lambda$. The quotient surface
$\mathbb C^*/ (z \sim \lambda z)$ through this identification is a
called a \textit{Hopf torus} and we call $\lambda$ its \textit{affine
  factor}.

\begin{figure}[!h]
  \includegraphics[width=0.3 \linewidth]{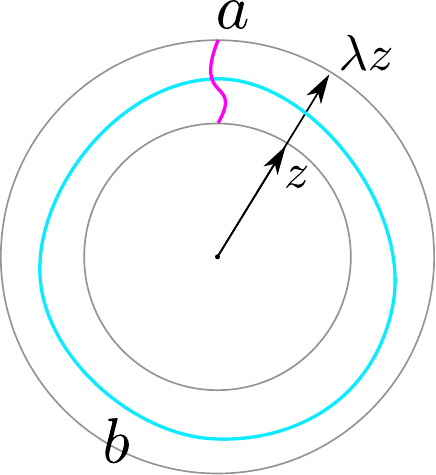}
  \caption{Hopf's torus}
  \label{Hopf}
\end{figure}

These provide a $1$-parameter family of affine structures on the
torus.  These surfaces have a very specific kind of
dynamics. Foliations in all directions have two closed leaves (the one
corresponding to the ray from zero in this direction) one of which is
attracting and the other repulsive.


Based on this torus, we can construct higher genus examples by gluing
two of these tori along a slit in the same direction (see Figure
\ref{Hopf}).  Take an embedded segment along the affine foliation in
one direction on one torus, and an other one in the same direction on
the second torus.  We cut the two surfaces along these segments and
identify the upper part of one with the lower part of the other with
the corresponding affine map.

\begin{figure}[!h]
  \includegraphics[width=0.6 \linewidth]{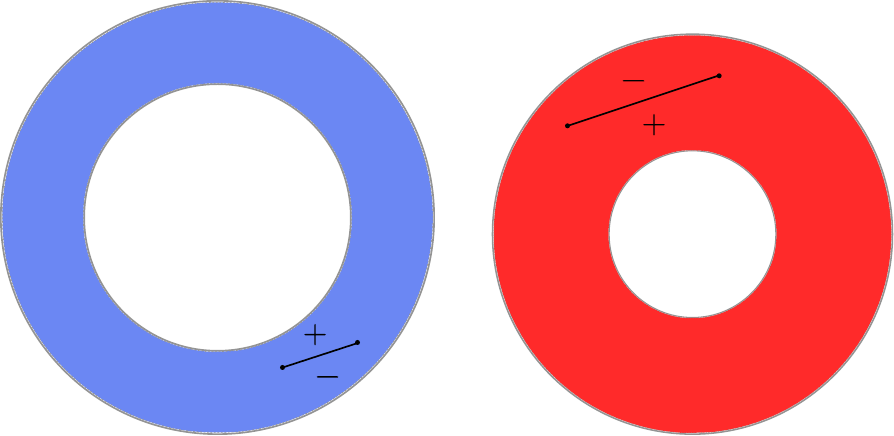}
  \caption{The franco-russian slit construction}
  \label{slit}
\end{figure}

Another construction based on the Hopf torus is given by considering a
finite covers.  Denote by $a$ the closed curve of the Hopf torus in
direction of the dilation and $b$ the closed curve turning around zero
once in the complex plane as in Figure \ref{Hopf}.  Take the $k$ index
subgroup of $\pi_1(T^2)$ generated by $a$ and $b^k$ and consider the
associated cover with the induced affine structure.  It is also a
torus which \textit{makes $k$ turns around zero}. We call it a
$k$-Hopf torus.  Similarly, a $\infty$-Hopf cylinder will be the cover
associated to the subgroup generated by $a$.

\begin{remark}
  We can also construct this structure as a slit construction along
  horizontal closed leaves of $k$ different Hopf tori.
\end{remark}

These structures have the remarkable property to be a disjoint union
of closed geodesics in different directions. Each of these geodesics
is an attractive leaf of the foliation in their direction.  This is
the property we want to keep track of, that is why we want to consider
angular sectors of these tori embedded in an affine surfaces. This is
the motivation for the following definition:

\begin{definition}
  Consider $\Sigma$ an affine surface.  Let $C_{\theta_2, \theta_1}$
  be an angular domain of a $\infty$-Hopf cylinder between
  \textit{angles} $\theta_1 \in \R_+^*$ and $\theta_1 < \theta_2$ such
  that $\theta_2 - \theta_1 = \theta$.  A \textit{cylinder} of angle
  $\theta$ is the image of a maximal affine embedding of some
  $C_{\theta_2, \theta_1}$ in $\Sigma$.

  \noindent We call $\lambda$ the \textit{affine factor} of the
  cylinder and $\theta$ its \textit{angle}.
\end{definition}

Remark the isomorphism class of an affine cylinder is completely
determined by the two numbers $\theta$ and $\lambda$.

\begin{proposition}
  \label{saddleco}

  Let $\Sigma$ be an affine surface and not a $k$-Hopf torus. Then the
  boundary of a maximal cylinder embedded in $\Sigma$ is a union of
  saddle connections.
\end{proposition}

\begin{proof}
  Suppose we embedded $C_{0,\theta}$ in the surface $\Sigma$.  If
  $\theta = \infty$ we would have a half-infinite cylinder in the
  surface. But this half-cylinder would have an accumulation point at
  $\infty$ in $\Sigma$ which contradicts its being embedded.

  When $\theta < \infty$, there are two reasons why $C_{0,\theta'}$
  cannot be embedded for $\theta' > \theta$
  \begin{enumerate}
  \item The embedding $C_{0,\theta} \rightarrow \Sigma$ extends
    continuously to the boundary of the cylinder in direction
    $\theta$, if the surface is not a $k$-Hopf torus the image
    contains a singular point. The image of the boundary is closed,
    and it is an union of saddle connections.
  \item The embedding does not extend to the boundary, then there is a
    geodesic $\gamma: [0,1) \rightarrow C_{0,\theta}$ starting close
    to the boundary and ending orthogonally in the $\theta$ boundary
    of $C_{0,\theta}$ such that $\gamma$ has no limit in $\Sigma$ when
    approaching $1$.  Consider an open disk in $C_{0,\theta}$ tangent
    to the boundary at the point to which $\gamma$ is ending and
    centered on $\gamma$ trajectory. Then Proposition
    \ref{closedgeodesic} implies that $\gamma$ starting from the
    center of the circle is a closed hyperbolic geodesic in $\Sigma$.
    This cannot happen since $\gamma$ is embedded in $\Sigma$
  \end{enumerate}
\end{proof}

Again a cylinder will be the union of closed leaves. The dynamics of a
geodesic entering such a cylinder is clear. If the cylinder is of
angle less than $\pi$ and the direction of the flow is not between
$\theta_1$ and $\theta_2$ modulo $2\pi$, then it will leave the
cylinder in finite time. Otherwise it will be attracted to the closest
closed leaf corresponding to its direction, and be trapped in the
cylinder.

As to enter a cylinder we have to cross its border, we see that for
cylinder of angle larger than $\pi$ every geodesic entering the
cylinder is also trapped.  These \textit{'trap'} cylinders can be
ignored when studying dynamics. We can study instead the surface with
boundaries where we remove all these cylinders. We will see in the
following section that these cylinders are also responsible for
degenerate behaviour when trying to triangulate the surface.

\begin{remark}
  A degenerate case of affine cylinder are flat cylinders. It is an
  embedding of the affine surface
  $C_a = \big\{ z \in \C \ | \ 0 < \Im(a) < a \big\} / (z \sim z+1)$.
  In this case, the length $a$ of the domain of the strip we quotient
  by $z \mapsto z+1$ will be called the \textit{modulus} of the
  cylinder.
\end{remark}

\subsection{Triangulations.}
\label{triangulations}

An efficient way to build affine surfaces is to glue the parallel
sides of a (pseudo-)polygon. A surface obtained this way enjoys the
property to have a \textit{geodesic triangulation}. It is a
triangulation whose edges are geodesic segments and whose set of
vertices is exactly the set of singular points. It is natural to
wonder if any affine surface has such triangulation from which we
could easily deduce a polygonal presentation. Remark that the question
only makes sense for surfaces of genus $g \geq 2$, for in genus $1$
there are no singular points. \\ Unfortunately, a simple example shows
that it is not to be expected in general. The double Hopf torus
constructed above cannot have a geodesic triangulation: any geodesic
issued from the singular point accumulates on a closed regular
geodesic, except for those coming from the slit. This obstruction can
be extended to any affine surface containing an affine cylinder of
angle $\geq \pi$: any geodesic entering such a cylinder never exits it
which is incompatible with the fact that a triangulated surface
deprived of its $1$-skeleton is a union of triangles. A remarkable
theorem of Veech proves that this obstruction is the only one:

\begin{theorem*}[Veech, \cite{Veech1,VeechU}]
  Let $\mathcal{A}$ be an affine structure which does not contain any
  affine cylinder of angle larger or equal to $\pi$. Then
  $\mathcal{A}$ admits a geodesic triangulation.
\end{theorem*}

The exact theorem generalizes a classical construction known as
\textit{Delaunay partitions} to a more general class of affine
surfaces. We review this construction and more of the material
contained in \cite{Veech1,VeechU} in Appendix \ref{Veech}.

\section{Examples.}
\label{examples}

\subsection{The two chambers surface.}

By gluing the sides of same color of Figure \ref{GMP} below, we get a
genus $2$ affine surface with a unique singular point of angle
$6\pi$. For completely random reasons, we call it the \GMP surface.

\begin{figure}[!h]
  \centering
  \includegraphics[scale=0.3]{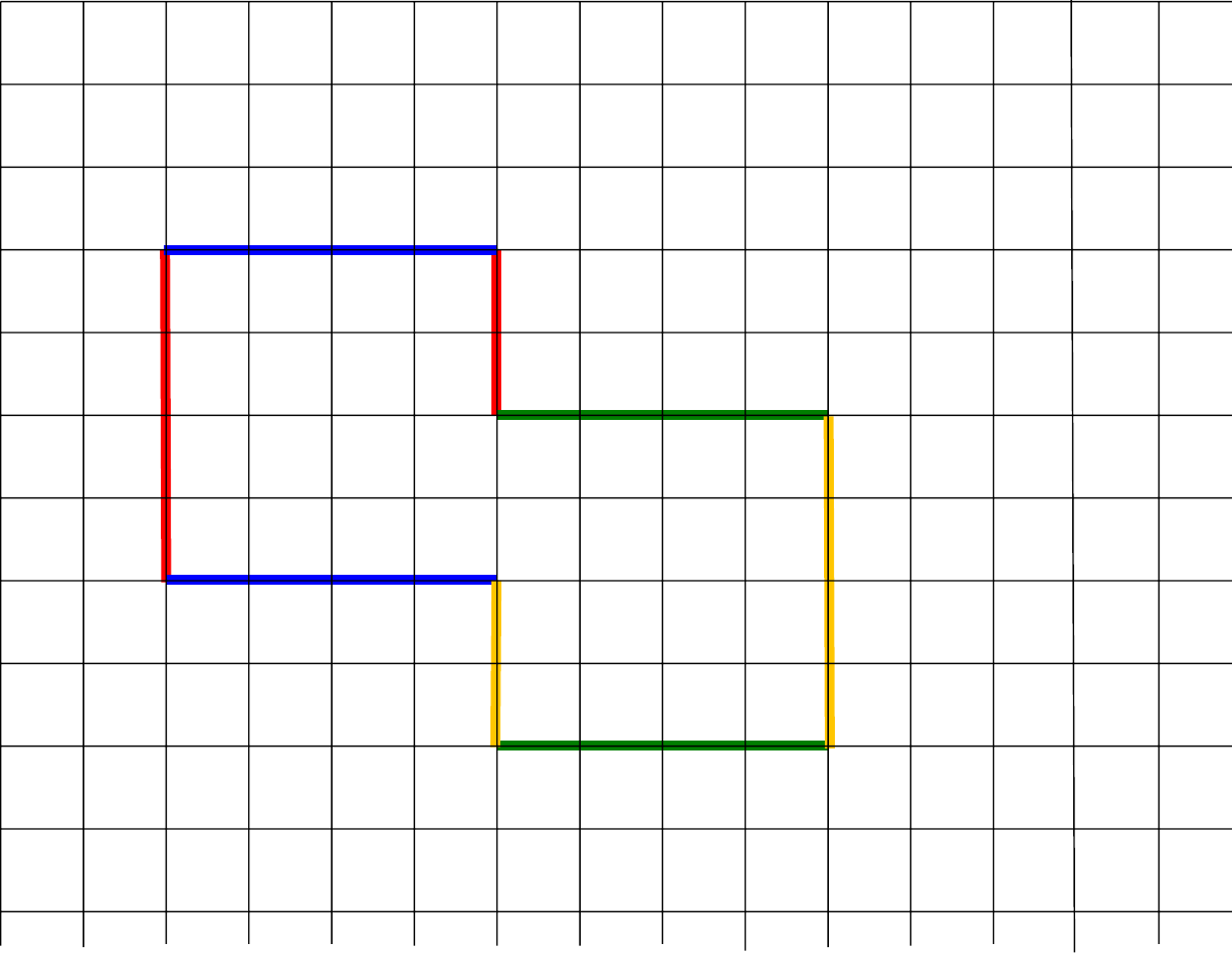}
  \caption{The \GMP \ surface.}
  \label{GMP}
\end{figure}

We believe that this example is particularly interesting because it is
a first non trivial example for which we can describe the directional
foliation in every direction.

\begin{proposition}

  Let $\mathcal{F}_{\theta}$ be the directional foliation oriented by
  $\theta$ on the \GMP \ example.

  \begin{enumerate}
  \item If $\theta = \pm \frac{\pi}{2}$, then the foliation is
    completely periodic; the surface decomposes into two euclidean
    cylinders.

  \item If $\theta =\arctan(n)$ or $\arctan(n+\frac{1}{2})$ for
    $n \in \mathbb{Z}$ and $\theta \neq \pm \frac{\pi}{2}$, the
    foliation accumulates on a closed saddle connection.

  \item For any other $\theta$, the foliation accumulates on a
    hyperbolic closed leave.
  \end{enumerate}

\end{proposition}

\begin{proof}
  Consider the case when $\theta = \arctan(\frac{1}{4})$. Here the
  segment linking the middle of the yellow sides projects to a closed
  leave of the directional foliation of angle
  $\theta = \arctan(\frac{1}{4})$. Take a little segment transverse to
  this leaf, the first return map is a dilation of factor
  $\frac{1}{2}$ (resp $2$). This closed leaf is therefore attractive
  in the sense that every leaf passing close by winds around and
  accumulate on it.

  The other cases are similar, the reader can convince himself by
  looking for the two closed leafs in the given direction and remark
  that one will be attractive and the other repulsive.
\end{proof}

\subsection{Affine interval exchange transformations.}
\label{Cantor}
We mention in this subsection a construction of Camelier and Gutierrez
(\cite{CamelierGutierrez}), improved by Bressaud, Hubert and Maas
(\cite{BressaudHubertMaass}), and generalised by Marmi, Moussa and
Yoccoz (\cite{MarmiMoussaYoccoz}).

\vspace{2mm}

An \textit{affine interval exchange} is a piecewise affine bijective
map from $[0,1]$ in itself. It can be thought of as a generalization
of either standard interval exchanges or of piecewise affine
homeomorphisms of the circle. To such an AIET (Affine Interval
Exchange Transformation), we can associate an affine surface which is
its \textit{suspension}. It consists in taking a rectangle,
identifying two vertical parallel sides in the standard way, and
identifying the two horizontal according to the AIET, Figure
\ref{DDuryev_cylinders}.

The dynamics of the vertical foliation of such an affine surface is
exactly the same as the affine interval exchange we started from, for
the orbits of the latter are in correspondence with the leaves of the
foliation, singular leaves corresponding to discontinuity points of
the AIET.

\vspace{2mm}

The construction mentioned above brings to light a surprising
behaviour for certain affine interval exchanges.

\begin{theorem*}[Marmi-Moussa-Yoccoz, \cite{MarmiMoussaYoccoz}]
  For all combinatorics of genus at least $2$, there exists a uniquely
  ergodic affine interval exchange whose invariant measure is
  supported by a Cantor set in $[0,1]$.
\end{theorem*}

The implication of the theorem for the affine surfaces associated is
that there are some leaves of the vertical foliation whose closure in
the surface is union of leaves which intersects every transverse curve
along a Cantor set. This in sharp contrast with both standard interval
exchanges and piecewise affine homeomorphisms of the circle.

The construction is quite involved and we will not give any
detail. Nonetheless it is worth noticing that in the example we have
presented before, we have seen no trace of such a behaviour: the \GMP
surface has a very simple dynamics in every direction.

\subsection{The \DbDu surface}

We build a first surface whose dynamics is \textit{a priori}
non-trivial. Choose $a,b$ two positive real numbers, we consider the
AIET associated to the permutation $(1,2)(3,4)$, and with top and
bottom lengths $a,b,b,a$.  Now take the suspension of this AIET with a
rectangle of height $1$, it defines an affine surface which two
singularities of angle $4\pi$ and genus $2$, see Figure
\ref{DDuryev_suspension}. We call it the \textit{\DbDu surface} and
denote it by $\DD_{a,b}$. Notice that the surface contains at hand
several affine cylinders. We represent some of them on the following
Figure \ref{DDuryev_cylinders}. These cylinders can overlap, and some
zones can a priori be without cylinder coverage.

\begin{figure}[h]
  \begin{subfigure}[c]{.4\linewidth}
    \centering
    \includegraphics[width=.9\linewidth]{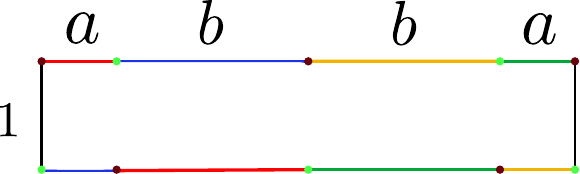}
    \subcaption{Suspension construction}
    \label{DDuryev_suspension}
  \end{subfigure}
  \begin{subfigure}[c]{.4\linewidth}
    \centering
    \includegraphics[width=.9\linewidth]{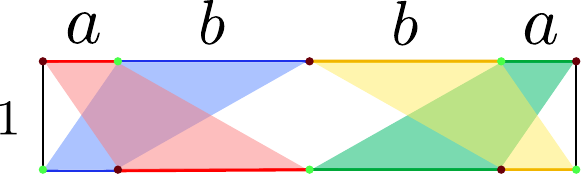}
    \subcaption{Cylinders}
    \label{DDuryev_cylinders}
  \end{subfigure}
  \caption{The \DbDu surface with a horizontal cylinder}
\end{figure}

We give an alternative representation of the surface which makes a
vertical flat cylinder decomposition appear.  To do so we cut out the
left part of the surface of width $a$ along a vertical line. We now
rescale it by a factor $\frac b a$ and reglue it on the top $b$
interval. Reproduce the same surgery with the right part of the
surface and the new surface is the one drawn on Figure
\ref{DDuryev_vertical}

\begin{figure}[!h]
  \includegraphics[scale=0.35]{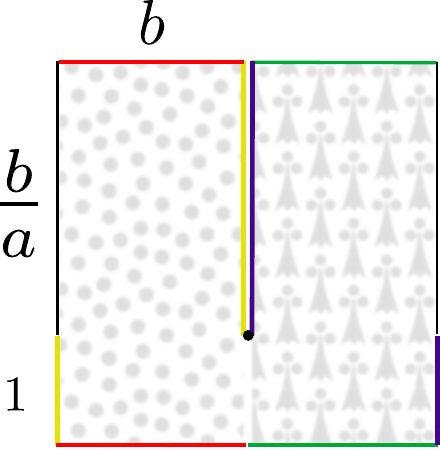}
  \caption{An alternative representation of the \DbDu surface}
  \label{DDuryev_vertical}
\end{figure}

\section{Veech groups.}
\label{FVgroups}

Given a matrix $M \in \mathrm{Gl}_2^+(\R)$ and an affine structure
$\mathcal{A}$ on $\Sigma$, there is a way to create a new affine
structure by replacing the atlas $(U_i, \varphi_i)_{i\in I}$ by
$(U_i, M.\varphi_i)_{i\in I}$. This new affine structure is denoted by
$M \cdot \mathcal{A}$. A way to put our hands on this operation is to
describe it when $\mathcal{A}$ is given by gluing sides of a polygon
$p$. If one sees $P$ embedded in the complex plane
$\C \simeq \mathbb{R}^2$, $M \cdot \mathcal{A}$ is the structure one
gets after gluing sides of the polygon $M \cdot P$ along the same
pattern.

We have defined this way an action of $\mathrm{Gl}_2^+(\R)$ on the set
of affine surfaces which factors through $\SLtR$, since the action of
dilation is obviously trivial. If $\mathcal{A}$ is an affine structure
on $\Sigma$, we introduce its Veech group $\FV(\mathcal{A})$ which is
the stabilizer in $\SLtR$ of $\mathcal{A}$, namely

$$ \FV(\mathcal{A}) = \{ M \in \SLtR \ | \ M \cdot \mathcal{A} = \mathcal{A} \} $$ 

The Veech group is the set of real affine symmetries of the considered
affine surface. It is the direct generalisation of the Veech group in
the case of translation surfaces (see \cite{HubertSchmidt} for a nice
introduction to the subject). For example, if $\mathcal{T}$ is a Hopf
torus, $\FV(\mathcal{T}) = \SLtR$. It is a consequence of the fact
that $\mathcal{T} = \C^* / (z \sim \lambda z)$ for a certain
$\lambda > 1$, and that the action of $\SLtR$ commutes to
$z \mapsto \lambda z$. This fact is in sharp contrast with the case of
translation surfaces where the Veech group is known to always be
discrete.

\subsection{The Veech group of the \GMP surface.}

We carry on the particular analysis of the examples introduced in the
last section, beginning with the \GMP surface. Our exhaustive
understanding of the dynamics of every directional foliation will
enable us to describe its Veech group completely.

\begin{proposition}

  The Veech group of the \GMP surface is the group generated by the
  two following matrices
  $$
  \begin{pmatrix}
    -1 & 0 \\
    0 & -1
  \end{pmatrix} \
  \text{and} \
  \begin{pmatrix}
    1 & 0 \\
    1 & 1
  \end{pmatrix}
  $$

\end{proposition}

\begin{proof}

  The are only two directions on the \GMP surface which are completely
  periodic which are $\frac{\pi}{2}$ and $-\frac{\pi}{2}$. Any element
  of the Veech group must preserve this set of directions, this is why
  it must lie in the set of lower triangular matrices. The rotation of
  angle $\pi$ belongs to the Veech group since both the polygon that
  defines the \GMP surface and the identification are invariant under
  this rotation. Up to multiplying by the latter, we can always assume
  that an element of the Veech group fixes the direction
  $\frac{\pi}{2}$.

  \vspace{2mm} We prove now that if a matrix of the form
  $\begin{pmatrix}
    \lambda & 0 \\
    * & -\lambda^{-1}
  \end{pmatrix} $ belongs to the Veech group with $\lambda > 0$, then
  $\lambda=1$. The \GMP surface contains two flat cylinders of modulus
  $1$. The image of these two cylinders by the action of such a matrix
  are two cylinders of modulus $1$.  So for $\begin{pmatrix}
    \lambda & 0 \\
    * & -\lambda^{-1}
  \end{pmatrix} $ to belong to the Veech of the \GMP surface,
  $\lambda$ must equal $1$.

  \vspace{2mm} Now remark that $\begin{pmatrix}
    1 & 0 \\
    1 & 1
  \end{pmatrix}$ is in the Veech group.  A simple cut and paste
  operation proves this fact, see Figure \ref{cutpaste} below.

\begin{figure}[h]
  \centering
  \includegraphics[scale=0.3]{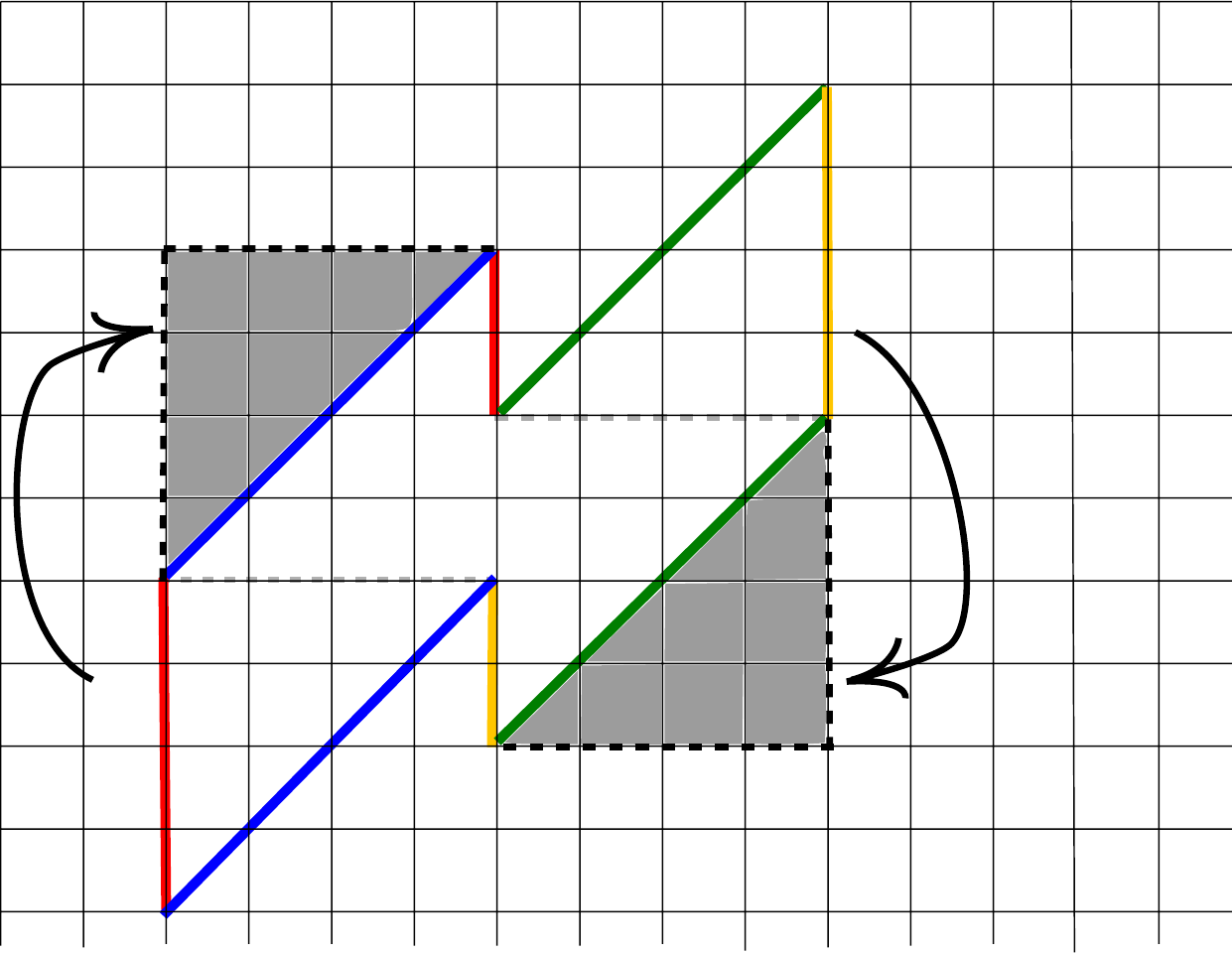}
  \caption{}
  \label{cutpaste}
\end{figure}

To complete the proof of the theorem, remark that the set of vectors
of saddle connection must be preserved. This implies that $t=1$ is the
smallest positive number such that
$\begin{psmallmatrix}1&0\\ t&1\end{psmallmatrix}$ belongs to the Veech
group of $\Sigma$.
\end{proof}

\subsection{The Veech group of the \DbDu surface.}

Let us describe some elements of the Veech group of the \DbDu surface.
First remark that when we act by the matrix
$$
\left(\begin{array}{cc}
        1 & t \\
        0 & 1 \\
      \end{array}
    \right)$$ on a vertical cylinder of height $1$ and width $t$, we
    can rearrange the surface and end up back to the same cylinder. It
    is exactly a Dehn twist on its core curve.  This works also for
    any cylinder of modulus $t$. Hence if we have a surface which we
    can decompose in cylinders of same modulus $t$ in the horizontal
    direction, the matrix above is in its Veech group.

    As we remarked when introducing the \DbDu surface $\DD_{a,b}$,
    they decompose into one cylinder of modulus $2(a+b)$ in horizontal
    direction (Figure \ref{DDuryev_cylinders}) and into two cylinders
    of modulus $\frac {1+b/a} b = \frac 1 a + \frac 1 b$ in vertical
    direction (Figure \ref{DDuryev_vertical}).

    As a consequence,
$$
\left(\begin{array}{cc}
        1 & 2(a + b) \\
        0 & 1 \\
      \end{array}
    \right), \; \left(\begin{array}{cc}
                        1 & 0 \\
                        \frac 1 a + \frac 1 b & 1 \\
                      \end{array}
                    \right) \in \FV(\DD_{a,b})
$$

Remark that these two matrices never generate a lattice in
$SL_2(\mathbb R)$ since it would imply that
$2(a+b)(\frac 1 a + \frac 1 b) \leq 4$ which never happens.

\subsection{Hopf surfaces}
\label{hopfsurfaces}
We present in this subsection a general construction of affine
surfaces whose Veech group is conjugate to

$$ \big\{   \begin{pmatrix}
  \lambda & t \\
  0 & \lambda^{-1}
\end{pmatrix} \ | \ t \in \R \ and \ \lambda \in \mathbb{R}_+^*
\big\} $$ and we prove that these are the only surfaces whose Veech
group is not discrete. In particular this construction includes the
Hopf torus and their derivatives introduced in Section
\ref{cylinders}.

\begin{definition}

  A \textit{ribbon graph} is a finite graph with a cyclic ordering of
  its semi-edges at its vertices.

\end{definition}

We can think of a ribbon graph as an embedding of a given graph in a
surface, the manifold structure giving the ordering a the
vertices. The structure of a tubular neighbourhood of the embedding of
the graph completely determines the ribbon graph. \\ Given a ribbon
graph, we can make the following construction: along the boundary
components of the infinitesimal thickening of the ribbon graph, we can
glue cylinders of angle $k\pi$ respecting the orientation of the
foliation to get an affine surface. We have to be careful to the
factors of the cylinders we glue if we want the singular points to be
Euclidean; we give an example from which it will be easy to deduce the
general pattern.

\begin{figure}[!h]
  \centering
  \includegraphics[width=\linewidth]{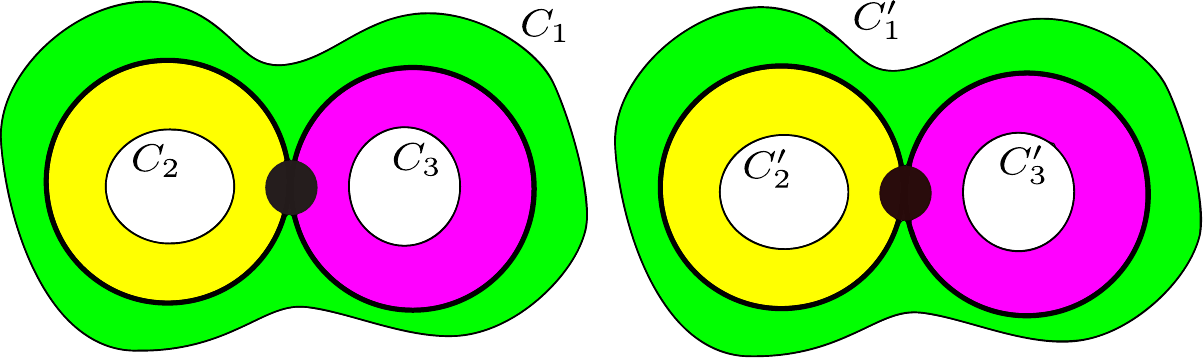}
  \caption{A ribbon graph with two vertices.}
  \label{cutting}
\end{figure}

\noindent that we turn into a genus two surface by gluing three
cylinders $D_i$ to the boundary components, each joining $C_i$ to
$C_i'$ for $i=1,2,3$.

\begin{figure}[!h]
  \centering
  \includegraphics[scale=0.5]{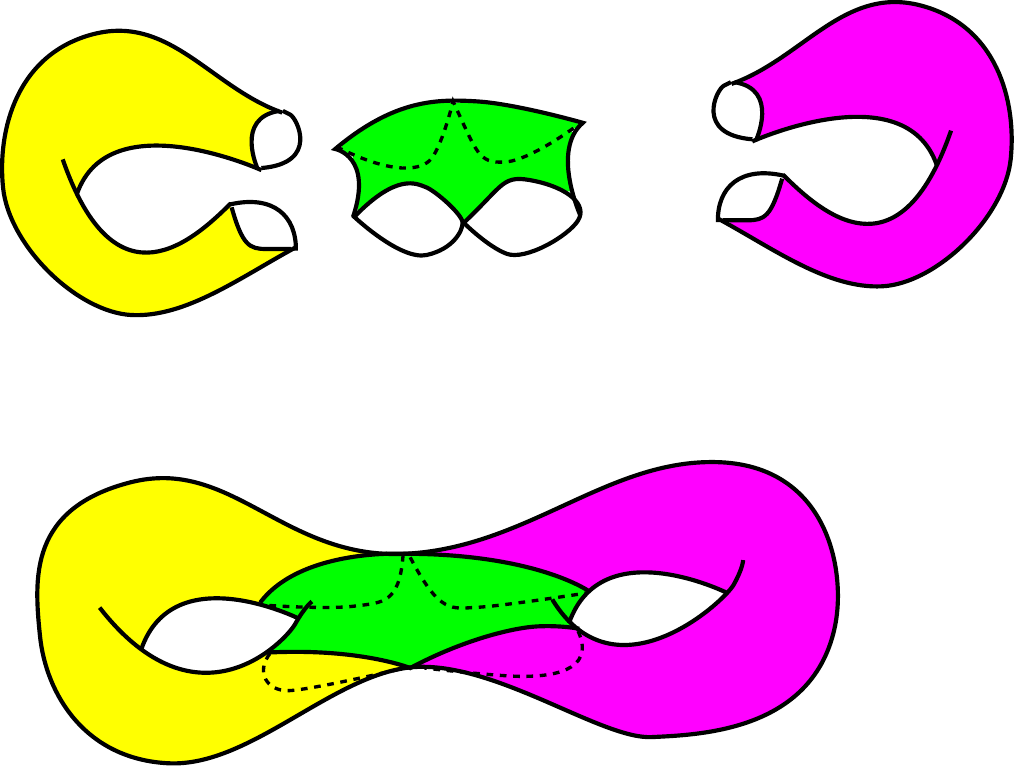}
  \caption{A cylinder decomposition of the surface of genus $2$.}
  \label{cutting}
\end{figure}

Let $\lambda_i$ be the affine factor of $D_i$. For the angular points
to be euclidean, it is necessary that the product of the factor of the
cylinders adjacent at a singular point is trivial. In our case we get
for an appropriate choice of orientation:

$$ \lambda_1 = \lambda_2 \lambda_3$$

This constraint straightforwardly generalizes to the general case and
is the only obstruction to complete the construction. We call an
affine surface obtained by this construction a \textit{Hopf surface},
because it generalizes the variations on the Hopf torus explained in
Section \ref{cylinders}.

We will see in the next section that this example corresponds exactly
to the case when the Veech group is not discrete and is not $\SL_2\R$

\subsection{Veech group dichotomy.}

First we deal with the case of genus $1$ with the following lemma,

\begin{lemma}
  \label{khopf}
  An affine torus is the exponential of some flat torus
  $\C/ \alpha \Z \oplus \left( \beta + 2ik\pi \right) \Z$ where $\alpha, \beta \in \R^+$
    and $k \in \mathbb N$.
  Moreover its Veech group is $\SLtR$
\end{lemma}

\begin{proof}
  Consider an affine torus. Its developing map goes from $\mathbb C$
  to $\mathbb C$ and its holonomy is commutative. Hence its holonomy
  is generated either by two translation or two affine maps with the
  same fixed point (which we assume to be zero).  The former case is
  to be excluded since we are only considering strictly affine
  surface.
  
  In the second case, we can choose a developing map $f$ which avoids
  zero, and associate to it the $1$-form
  $\operatorname d \log f = \frac {df} {f}$.  As the surface has no
  singularity, the derivative of $f$ is never zero, and the $1$-form
  is invariant with respect to the holonomy.  Thus this \textit{a
    priori} meromorphic form is defined on the torus, has no zeroes
  and by residue formula no poles, therefore it is holomorphic.

  In conclusion, the logarithm form gives a flat structure on the
  torus that is isomorphic to
  $\C/ \alpha \Z \oplus \tau \Z$ where
  $\alpha \in \R^+$ and $\tau \in \C^*$.  The exponential of
  this flat torus is the initial affine structure, thus
  $e^\tau \in \R^*_+$ and $\Im(\tau) \equiv 0 \mod{2\pi}$.

  Any matrix of $\SLtR$ commutes with the scalar multiplications, thus
  the Veech group of such a surface is the whole $\SLtR$.
\end{proof}

This structure is like taking one $\alpha$-Hopf torus which we slit
at one horizontal closed curve and glue $k$ copies of it. When we glue
back the $k$-th copy to the first one, we apply a $\beta$ dilation.\\

Now that we have set aside the peculiar case of genus $1$, we prove a
classification theorem on affine surfaces of higher genus depending on
the type of their Veech group.

\begin{theorem}
  \label{classification}
  Let $\Sigma$ be an affine surface of genus $\geq 2$.  There are two
  possible cases :

  \begin{enumerate}
  \item $\Sigma$ is a Hopf surface and $\FVS$ is the subgroup of upper triangular elements of $\SLtR$,
    $$
    \big\{
    \begin{pmatrix}
      \lambda & * \\
      0 & \lambda^{-1}
    \end{pmatrix}
    | \lambda \in \R^* \big\}
    $$
  \item $\FVS$ is discrete.
  \end{enumerate}
\end{theorem}

The end of the section is devoted to proving Theorem
\ref{classification}.  To do so, we will distinguish between affine
surfaces having saddle connections in at least two directions and
those who do not. The former enjoy the property that their Veech group
is automatically discrete thanks to a classical argument inspired by
the case of translation surfaces (see Section 3.1 of
\cite{HubertSchmidt}). The latter will turn out to be Hopf surfaces
introduced in section \ref{hopfsurfaces}.

\begin{lemma}
  \label{discretness}
  Let $\Sigma$ be an affine surface having two saddle connections in
  different directions.  Then $\FVS$ is a discrete subgroup of
  $\SLtR$.
\end{lemma}

\begin{remark}
  As a direct corollary both the \GMP surface and the \DbDu surface
  have a discrete Veech group.
\end{remark}

\begin{proof}
  Consider $\FVSo$ the subgroup of $\FVS$ fixing point-wise the set of
  singular points of $\Sigma$. This subgroup has finite index in
  $\FVS$ and its being discrete implies discreteness for $\FVS$.  \\
  Choose an arbitrary simply connected subset $U \subset \Sigma$
  containing all the singularities and an arbitrary developing map of
  the affine structure on $U$.  Thanks to this developing map, we can
  associate to each oriented saddle connection a vector in $\R^2$. The
  set of such vectors enjoys two nice properties:

  \begin{itemize}
  \item it is discrete;
  \item it is invariant under the action of $\FVSo$.
  \end{itemize}

  We have made the hypothesis that a pair of those vectors
  $(v_1, v_2)$ form a basis in $\R^2$.  If $a_n$ is a sequence of
  elements of $\FVSo$ going to identity and $A_n$ the matrices of
  their action in the normalization induced by $U$ (notice that an
  element of $\FVS$ is an element of the quotient
  $\mathrm{Gl}^+_2(\R)/ \R^*_+$ and therefore the matrix of its action
  depends on normalization both at the source and at the target).
  $A_n \cdot v_1 \rightarrow v_1$ and $A_n \cdot v_2 \rightarrow v_2$
  but the discreteness of the set of saddle connection vectors implies
  that for $n$ large enough $A_n \cdot v_1 = v_1$ and
  $A_n \cdot v_2 = v_2$, and thus $A_n = \mathrm{Id}$. Which proves
  the discreteness of the Veech group.
\end{proof}

We will now characterize the affine surfaces having saddle connections
in at most one direction. The two following lemmas will complete the
classification.

\begin{lemma}
  If all the saddle connections of an affine surface $\Sigma$ are in
  the same direction it is a Hopf surface.
\end{lemma}

\begin{proof}

  We make the assumption that $\Sigma$ has at least a singular point
  (if not, Lemma \ref{khopf} settles the question).  We are going to
  prove that every singular point has a least a saddle connection in
  every angular sector of angle $\pi$. An affine surface whose all
  separatrices in one direction are saddle connection is easily seen
  to be a Hopf surface, see Subsection \ref{hopfsurfaces}.

  \vspace{2mm}

  Consider the exponential map at a singular point $p$ associated to a
  local affine normalization and let $r>0$ be smallest radius such
  that $\Delta_r$ the (open) semi-disk of radius $r$, bounded below by
  two horizontal separatices and containing a vertical separatrix,
  immerses in $\Sigma$ by means of the exponential map. If
  $r= \infty$, we would have a maximal affine immersion of
  $\mathcal{H}$, whose boundary would project to a closed leaf
  containing $p$ and there would be two saddle connections in the
  horizontal direction. Otherwise $r < \infty$. There can be two
  different reasons why $\Delta_{r'}$ does not immerse for $r'>r$:

  \begin{enumerate}

  \item either the immersion $\Delta_r \longrightarrow \Sigma$ extends
    continuously to the boundary of $\Delta_r$ and the image of this
    extension contains a singular point. In that case this singular
    point must be in the unique direction containing saddle
    connections (say the vertical one) and the vertical separatrix was
    actually a saddle connection; \sk

  \item or $\Delta_r \longrightarrow \Sigma$ does not extend to the
    boundary of $\Delta_r$.

  \end{enumerate}

  We prove that the latter situation cannot occur. Otherwise there
  would be a geodesic $\gamma$ issued at $p$ affinely parametrized by
  $[0,r)$ such that $\gamma(t)$ does not have a limit in $\Sigma$ when
  $t$ tends to $r$. We make the confusion between $\gamma$ as a subset
  of $\Sigma$ and its pre-image in $\Delta_r$.  Considering an open
  disk in $\Delta_r$ which is tangent at the boundary to the point in
  $\partial \Delta_r$ towards which $\gamma$ is heading and such that
  its center belongs to $\gamma$. By means of the immersion of
  $\Delta_r$ in $\Sigma$, it provides an affine immersion of $\D$
  satisfying the hypothesis of Proposition
  \ref{closedgeodesic}. $\gamma$ would therefore project to a closed
  hyperbolic geodesic, which contradicts its being a separatrix.

\end{proof}

\subsection{The Veech group of an affine surface is probably not a
  lattice.}

We start by showing the following lemma,
\begin{lemma}
  If $\FVS$ is discrete it cannot be cocompact.
\end{lemma}

\begin{proof}
  To show this we only need to find a continuous function on
  $\Hy / \FVS$ that is not bounded. For flat surface the systole does
  the trick, here we cannot define length of saddle connections but
  the ratio of the two lengths of saddle connections that intersect.
  Thus we can define the shortest and the second shortest simple
  saddle connections starting at a given singular point, and we denote
  by $l$ and $L$ their length in one arbitrary chart around the
  singularity. These two distinct saddle connections exist since we
  assume $\FVS$ discrete according to Theorem \ref{classification}.
  Now the ratio $L/l$ is independent from the previous choice of
  chart.  And if we take the minimum value of this ratio on all the
  singularities, it will be a continuous function on the $\Hy$ orbit
  of the surface invariant by the Veech group.

  To finish the argument, apply the Teichmüller deformation of the
  surface with matrices of the form
  $$
  \left(
    \begin{array}{c c}
      e^t & 0      \\
      0   & e^{-t}
    \end{array}
  \right)$$ where the horizontal direction will be the direction of
  the smallest saddle connection with the smallest ratio.  This
  deformation will decrease the length $l$ and increase $L$ as $t$
  goes to infinity and thus make this function go to infinity.
\end{proof}

As we saw in Section \ref{cylinders}, cylinders trap the linear flow
in their corresponding angular sector. This behaviour restricts the
potential directions for saddle connections around a singularity in
the boundary of cylinders and prevent
Veech groups from being lattices.\\

\begin{figure}[h]
  \includegraphics[width=.2\linewidth]{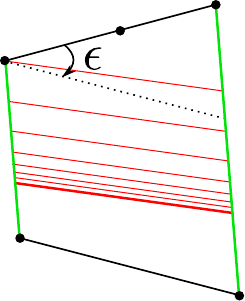}
  \caption{Angular section on which leaves are hyperbolic}
  \label{angular}
\end{figure}

Consider $\Sigma$ an affine surface endowed with an affine cylinder,
take any singular point at the border of this cylinder and flow a leaf
heading inside the cylinder whose angular direction falls just
in-between the two extreme angles of the cylinders.  This leaf will be
trapped inside the cylinder and accumulate to a close leaf (Figure
\ref{angular}).  As a consequence none of these leaves will meet a
singular point. There won't be any saddle connection starting from the
chosen point in the angular sector
define by the cylinder.\\

This implies the following proposition :

\begin{proposition}
  \label{nolattice}
  If $\Sigma$ is an affine surface with an affine cylinder then $\FVS$
  is not a lattice.
\end{proposition}

\begin{proof}
  Assume for a contradiction that $\FVS$ is a lattice. Take a finite
  index subgroup which stabilizes all the singularities of the
  surface.  As an assumption on $\Sigma$ there is an affine cylinder
  in the surface, Proposition \ref{saddleco} implies that there is a
  singularity in the boundary of the cylinder, and a small angular
  section in which any separatrix from this singularity will
  accumulate to a closed leaf.  As the subgroup is a lattice it
  contains a parabolic element and its limit set is the whole border
  of $\mathbb H$.  Then we can conjugate this parabolic element to see
  that there is a dense set of direction in which there is a
  parabolic. In parabolic directions, the separatrices are all saddle
  connections. Indeed, if not there would be an accumulation point,
  and a neighborhood of this point would be crossed infinitely many
  times by the separatrix on which the parabolic acts as the
  identity. Hence the parabolic element would act as the identity on
  the whole neighborhood, which is a contradiction.  This shall not be
  since we showed that there cannot be saddle connection in an open
  set of directions around the singularity.
\end{proof}

\section{Cylinders on genus 2 surfaces.}
\label{thmgenus2}

The purpose of this section is to show the following result

\begin{theorem}
  \label{cylinderingenus2}
  Any affine surface of genus $2$ has a cylinder.
\end{theorem}

First, remark that Veech's theorem on Delaunay triangulation (see
Veech's theorem in \ref{triangulations}) tells us that if an affine
surface does not have a triangulation, it must contain an affine
cylinder (of angle at least $\pi$). We can therefore forget about this
case and assume that all the surfaces we are looking at have geodesic
triangulations.  \\ The theorem is also easy to prove when the surface
has two singular points which both must be of angle $4\pi$. Consider a
triangle of a Delaunay triangulation of the surface, at least two of
its vertices are equal to the same singular point and the side
corresponding to these two vertices is a simple closed curve. It must
cut the angle of the associated singular point into two angular
sectors of respective angle $3\pi$ and $\pi$. Which implies that it
bounds a cylinder on the side of the angle $\pi$.

\vspace{2mm} For the remainder of the section, $\Sigma$ is a surface
of genus $2$ together with a strictly affine structure whose unique
singular point of angle $6\pi$ is denoted by $p$. A geodesic
triangulation of $\Sigma$ must have exactly 9 edges and 6
triangles. In this particular case where the triangulation has a
unique vertex, each edge defines a simple closed curve, geodesic away
from $p$ and cutting the latter into two angular sector. Because each
directional foliation is oriented, such an edge must cut the angle
$6\pi$ into two angles of respective values either $5\pi$, $\pi$ or
$3\pi$, $3\pi$.  \\ The lemma below proves that any such geodesic
triangulation has at least one edge cutting the singular point into
two angles $5\pi$ and $\pi$. The theorem is a direct consequence of
this lemma.

\begin{lemma}
  \label{combinatoire}
  A geodesic triangulation of $\Sigma$ cannot have its $9$ edges
  cutting $p$ into two sectors of angles $3\pi$.
\end{lemma}

\begin{proof}
  The property on the edges implies that they all intersect with
  number $\pm 1$ at the only vertex of the triangulation.  Since every
  closed leaf separates the cylinder at $p$ in two angular sectors of
  angle $3\pi$ an oriented leaf can't have it's in-going and outgoing
  parts in the same sector cut by any other leaf (the angle would have
  to be smaller).

  We can always see topologically our surface as an octagon which
  boundary are edges of the triangulation. Take for example the
  maximal sub-graph of the $1$-skeleton of the triangulation such as
  the complementary of it in the surface is connected and
  simply-connected, this complement will be the fundamental domain we
  are looking for.  Now the intersection number property tells us that
  the configuration of the path at the order will be in the setting of
  Figure \ref{separatrix}

  \begin{figure}[h]
    \begin{subfigure}[c]{.4\linewidth}
      \centering
      \includegraphics[height=3cm]{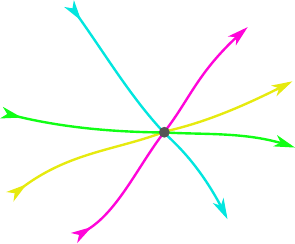}
    \end{subfigure}
    \begin{subfigure}[c]{.5\linewidth}
      \centering
      \includegraphics[height=3.5cm]{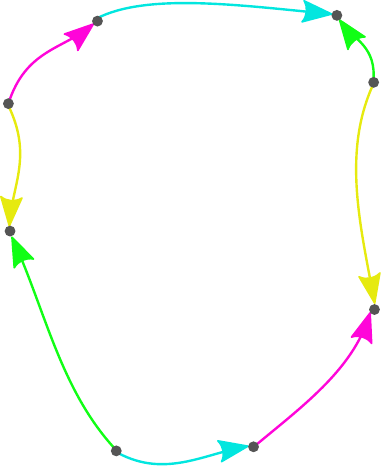}
    \end{subfigure}
    \caption{Topological setting of the separatrix diagram}
    \label{separatrix}
  \end{figure}

  Now consider yet another of the $9$ edges, as it has to intersect
  all of the other path with $\pm 1$ the only possibility is when the
  curves starts between two colors and end up between the two same
  colors.

  Let's add one of these edges as in Figure \ref{separated}.  It is
  clear now that we won't be able to add any curve with the same
  property since they can't intersect a curve in other point than $p$.

  \begin{figure}[h]
    \begin{subfigure}[c]{.4\linewidth}
      \centering
      \includegraphics[height=3cm]{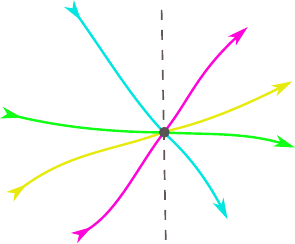}
    \end{subfigure}
    \begin{subfigure}[c]{.5\linewidth}
      \centering
      \includegraphics[height=3.5cm]{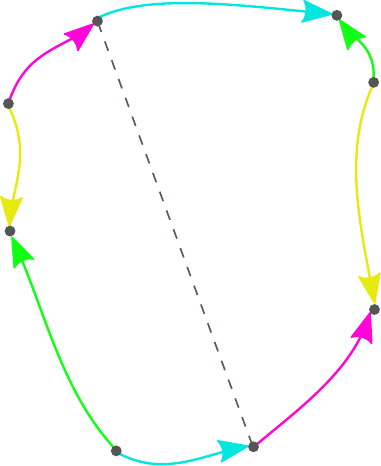}
    \end{subfigure}
    \caption{Adding one curve}
    \label{separated}
  \end{figure}

\end{proof}

This completes the proof of Theorem \ref{cylinderingenus2}. It is not
completely satisfactory: we would like to prove that every affine
surface contains \textit{strictly affine cylinders}.

\section{Open problems.}
\label{open}
We hope that at this point, most of the problems we are about to
suggest seem natural to the reader.

\vspace{2mm}

We have so far put our hands on several dynamical behaviours for the
directional foliations on our affine surfaces. Very often it happens
that a finite number of \textit{hyperbolic closed leaves} attract all
the others, as it the case for all but one direction on the \GMP
surface. We say in that case that the dynamics is \textbf{hyperbolic}.
\noindent The Camelier-Gutierrez construction discussed in Section
\ref{Cantor} also proves the existence of directional foliations such
that every leaf accumulates to a closed union of leaves which is
transversely a \textit{Cantor set}. Finally it is not to be excluded
that some directional foliations are minimal (it is actually very easy
to build examples of such affine surfaces). Because of the conjectural
picture we are about to draw, we call both late cases
\textbf{exceptional}.

We begin our list of open problems by very specific questions
concerning the $\DD_{1,2}$ example. It is a simple and very explicit
one, but the dynamical questions that it raises are not
straightforwardly answerable.

\begin{enumerate}

\item Does there exist a dense direction on the $\DD_{1,2}$ example?

\item Does there exist a 'Cantor like' direction on the $\DD_{1,2}$
  example?

\item What does the set of hyperbolic directions on the $\DD_{1,2}$
  example look like? Is it dense? Has it full Lebesgue measure?

\end{enumerate}

These might be a good starting point on the way to the general
case. Although the combinatorial arguments used in its proof are
unlikely to generalize to higher genus, Theorem \ref{cylinderingenus2}
suggests that systematic dynamical behaviours are to be expected. Is
it true that

\begin{enumerate}[resume]

\item every affine surface has a closed regular geodesic?

\item every affine surface has a closed, regular and hyperbolic
  geodesic?

\item the set of hyperbolic direction of every affine surface is
  dense? has full measure?
\end{enumerate}

These questions have natural generalizations to the moduli space of
affine surfaces together with a directional foliation. It is possible
that in this setting some of the aforementioned questions are easier
to answer, and that specific surfaces have a very different behaviour
from the generic one. We conjecture that the answers to these three
questions are positive, with a greater reserve about the full measure
one.  There are most probably very interesting things to say about the
loci of exceptional directions, but making guesses without a clear
picture of what happens for hyperbolic ones seems quite dairy.
\vspace{2mm} A very natural direction at this point is to investigate
the geometrical properties of these affine surfaces:

\begin{enumerate}[resume]

\item Which are the affine surfaces having only finitely many saddle
  connections?

\item Is it true that trough a given point on an affine surface always
  passes either a closed geodesic or a saddle connection? It is the
  case in the \GMP example.

\item What does the set of vectors of saddle connections on a given
  surface look like?
\end{enumerate}

Proposition \ref{nolattice} most probably prevents the Veech group of
an affine surface to be a lattice in $\SLtR$. Nonetheless, it seems to
be an interesting invariant of these affine surfaces.

\begin{enumerate}[resume]
\item What kind of Fuchsian groups can appear as Veech groups ?
\end{enumerate}

\vspace{2mm} Finally we want to suggest that a nice source of
questions is trying to describe the set of surfaces having specific
interesting dynamical/geometric properties. For instance:

\begin{enumerate}[resume]

\item Is the set of affine surfaces having an exceptional direction
  dense? generic?

\item Is the set of affine surfaces having \textit{no} exceptional
  direction dense? generic?

\item Does there exists a surface having infinitely many
  \textit{'Cantor like'} directions?

\end{enumerate}
 
\appendix 
\section{Veech's results on the geometry of affine surfaces.}
\label{Veech}

We review in this appendix results of Veech on the geometry of affine
surfaces appearing in \cite{Veech1} and \cite{VeechU}. Note that Veech
works in the more general context of affine surfaces with
singularities.

For the sake of clarity, we will restrict his results to the case
under scrutiny in this paper namely branched affine surface with real
positive linear holonomy. The notes \cite{VeechU} remain unpublished
and Veech kindly allowed us to reproduce here proofs that are
contained in these notes.

\subsection{The property $\mathcal{V}$.}

We say an affine surface $\Sigma$ satisfies \textit{the property}
$\mathcal{V}$ if there is no affine immersion of $\mathbb{H}$ in
$\Sigma$. It is equivalent to ask that $\Sigma$ has no affine
cylinder of angle larger than $\pi$.

\begin{theorem*}[Veech, \cite{VeechU}]
  An equivalent formulation of the property $\mathcal{V}$ is the
  following:

  $(\mathcal{V}')$ Every affine immersion of the open unit disk
  $\mathbb{D} \subset \C$ in $\Sigma$ extends continuously to a map
  $\overline{\mathbb{D}} \rightarrow \Sigma$.
  
\end{theorem*}

We will only give the proof of one sense of the equivalence, namely
that the property $\mathcal{V}$ implies the property $\mathcal{V}'$,
which will be sufficient for our purpose.

\begin{lemma}
  \label{extension1}

  Let $\varphi$ be an affine immersion of the open unit disk
  $\mathbb{D} \subset \C$ in $\Sigma$ that does not extend
  continuously to a map $\overline{\mathbb{D}} \rightarrow
  \Sigma$. Then $\varphi$ extends to an affine immersion
  $\mathbb{H} \longrightarrow \Sigma$.
\end{lemma}

\begin{proof}
  Since $\varphi$ does not extend to $\partial \D$, there exists
  $z \in \partial D$ such that
  $$ \lim_{t \rightarrow 1}{\varphi(tz)} $$ does not exist. Let
  $\gamma$ be the path $t \mapsto \varphi(tz)$. Let $x$ be an
  accumulation point of $\gamma$. Since $[0,1]$ is connected, we can
  assume that $x$ is not singular. Let $t_k \rightarrow 1$ be an
  increasing sequence such that $\gamma(t_k) \rightarrow x$ and $f$ an
  affine chart at $x$, $f : U \mapsto \Delta$ where $\Delta$ is the
  unit disk centered at $0$ et $U$ an open set containing $x$ such
  that $f(x) = 0$. We denote by $\frac{1}{2} U$ the pre-image by $f$
  in $\Sigma$ of the disk centered at $0$ or radius $\frac{1}{2}$.

  For $k$ large enough, we can assume that $\gamma(t_k)$ belongs to
  $\frac{1}{2} U$. Denote by $V_k$ the pre-image by in $\Sigma$ of the
  disk of radius $\frac{1}{2}$ centered at $\varphi(\gamma(t_k))$. In
  particular $V_k$ contains $x$. We claim that the image under
  $\varphi$ of $D_k$ the open disk centered at $t_kz \in \mathbb{U}$
  tangent to $\partial U$ at $z$ contains $V_k$. Since $\varphi$ does
  not extend at z, there is a closed disk centered a $t_kz$ strictly
  contained in $D_k$ whose image under $\varphi$ is not contained in
  $V_k$. Since this image is a disk concentric at $t_kz$ it must
  contain $V_k$ and in particular $x$. Let $w_k$ be a pre-image of $x$
  in $D_k$.

  Let $E_k$ be the largest disk with center $w_k$ to which $\varphi$
  admits analytic continuation. Since $w_k \rightarrow z$, the radius
  of $E_k$ converges to $0$. Necessarily $\varphi(E_k)$ contains
  $U = f^{-1}(\Delta)$ because $\varphi(E_k)$ is a maximal embedded
  disk of center $x$. Extending the map
  $f^{-1} : \Delta \rightarrow U \subset \Sigma$ by means of $\varphi$
  defines $F_k : \Delta_k \rightarrow \Sigma$. The functions
  $(F_k)_{k \in \mathbb{N}}$ have the following properties:

  \begin{itemize}

  \item $\Delta_k$ is a disk;

  \item $\Delta \subset \Delta$;

  \item $\forall \zeta \in \Delta$ we have that
    $F_k(\zeta) = f^{-1}(\zeta)$;

  \item $\forall k, k' \in \mathbb{N}$, $F_k = F_{k'}$ on
    $\Delta_k \cap \Delta_{k'}$;

  \item the radius of $\Delta_k$ tends to $+\infty$, because the
    radius of $E_k$ tends to $0$.

  \end{itemize}

  Since all the $D_k$ are connected and
  $\bigcap_{k \in \mathbb{N}}{\Delta_k}$ is non empty, the $F_k$
  define an affine immersion

  $$ F : \bigcup_{k \in \mathbb{N}}{\Delta_k} \longrightarrow \Sigma $$

  \noindent and $\bigcup_{k \in \mathbb{N}}{\Delta_k}$ must contain a
  half-plane since it is the union of disk whose radius tend to
  infinity whose intersection is not empty. This proves the lemma.

\end{proof}

\begin{lemma}
  \label{extension2}
  Any affine immersion $ \varphi :\mathbb{H} \longrightarrow \Sigma$
  can be extended to $ \varphi' :\mathbb{H}' \longrightarrow \Sigma$
  such that the latter is invariant by the action by multiplication by
  a positive real number $\lambda \neq 1$.
\end{lemma}

\begin{proof}
  We show first that such an immersion cannot be one-to-one. Consider
  the geodesic
  $\gamma [0, \infty[ \rightarrow [i, i\infty[ \subset \Hy$. Its image
  by $\varphi$ cannot have a limit in $\Sigma$ for it has infinite
  (relative) length. Let $x \in \Sigma$ be an accumulation point of
  this image, and $U \subset \Sigma$ a closed embedded disk with
  center $x$ such that $\varphi(\gamma)$ is not completely contained
  in $U$. Since $\varphi(\gamma)$ is a leaf of a directional foliation
  on $\Sigma$, it must cross $U$ twice along parallel segment and
  therefore cannot be injective because the image of horizontal
  stripes at the first and second crossing will overlap.

  \vspace{2mm}

  There exists then $v \neq w \in \Hy$ such that
  $\varphi(z) = \varphi(w)$. There exists then an affine map
  $z \mapsto \lambda z + a $ with $a \in \R^*_+$ and $b$ in $\C$ such
  that $w = \lambda v + a$ and for all $z$ in a neighborhood of $v$

$$ \varphi(z) = \varphi(\lambda z + a).$$

The union of the iterated images of $\Hy$ by
$z \mapsto \lambda z + a $ is a half-plane $\Hy'$ to which $\varphi$
extends by analytic continuation and on which the above invariance
relation holds for all $z$. If $\lambda = 1$, the image of $\varphi'$
in $\Sigma$ is an infinite flat cylinder, which is impossible since
$\Sigma$ is compact; therefore $\lambda \neq 1$. The fixed point of
$z \mapsto \lambda z + a $ must lie in $\partial \Hy'$ and up to an
affine transformation we can suppose that $b = 0$.

\end{proof}

\subsection{Delaunay decompositions and triangulations.}

The main result of \cite{Veech1} and \cite{VeechU} combined is
establishing the existence of geodesic triangulations for surfaces not
verifying the only natural obstruction for this to happen. It is given
by the following theorem:

\begin{theorem*}[Veech]
  An affine surface admits a geodesic triangulation if and only if it
  does not contain an embedded open affine cylinder of angle
  $\pi$. Equivalently, an affine surface admits a geodesic
  triangulation if and only if it satisfies the property
  $\mathcal{V}$.
\end{theorem*}

This theorem is a corollary of the existence of \textit{Delaunay
  partitions} for affine surfaces that Veech deals with. We explain in
the sequel the construction. Let $\Sigma$ be an affine surface
satisfying the property $\mathcal{V}$ and let $x \in \Sigma$ be a
regular point. We are going to distinguish points depending on the
number of singular points on the boundary of the largest immersed disk
at $x$. We denote this number by $\nu(x)$. The set
$\big\{ \nu(x) =1 \big\}$ is a open dense set in the surface. The
special points a those such that $\nu(x) \geq 3$, who form a discrete
and therefore finite set in the surface. At these points one can
consider the largest embedded disk in the surface and consider in this
disk the convex hull of the (at least) three singular point on the
boundary. The crucial (but not completely obvious) facts are that:

\begin{itemize}

\item this convex hull projects onto an \textbf{embedded} convex
  polygon in the surface;

\item the union of such polygons covers the whole surface;

\item such polygons only intersects at their boundary;

\item an intersection of two such polygons is union of some their
  (shared) sides;

\item the set of vertices of such polygons is exactly the set of
  singular points;

\item the union of the interior of these polygons is exactly the set
  $\big\{ \nu(x) =1 \big\}$;

\item the union of the interior of their sides is exactly the set
  $\big\{ \nu(x) =2 \big\}$.
  
\end{itemize}

This decomposition of the surface in convex polygons is called its
\textbf{Delaunay polygonations}, is unique and only depends on its
geometry. A triangulation of each polygon leads to a \textit{geodesic
  triangulation} of the surface.

\vspace{2mm} The crucial fact (we refer to \cite{Veech1} for
details) is that having the property $\mathcal{V}$ implies that
maximal affine embeddings of the disk extend to their boundary, which
is the technical point that one needs to make sure the construction
hinted above can be carried on.

\begin{remark}
  The converse of the triangulation theorem is quite easy. Any
  cylinder of angle at least $\pi$ behave like a 'trap': any geodesic
  entering it never escapes. A surface containing such a cylinder
  therefore cannot be geodesically triangulated, because no edge of
  the triangulation could enter the cylinder and the complement of the
  $1$-skeleton of such a triangulation would not be a union of cells.
\end{remark}

\subsection{Closed geodesics and wild immersions of the disk. }

We consider in this subsection a surface $\Sigma$ which does not
satisfy the property $\mathcal{V}$. This implies that there is an
affine immersion $\varphi : \mathbb{D} \longrightarrow \Sigma$ which
does not extend to $\pD$ the boundary of $\D$. This also implies the
existence of an affine cylinder and therefore closed hyperbolic
geodesics. We give in this subsection a way to localize such a closed
hyperbolic geodesic starting from $\varphi$. We believe that this
result should be attributed to Veech. Even though not stated
explicitly in \cite{VeechU} (probably because Veech did not have in
mind the dynamical questions we are interested in), the statement is
obvious for anyone who has read carefully the sections 23 and 24 of
\cite{VeechU}.

\begin{proposition}[Veech, \cite{VeechU}]
  \label{closedgeodesic}
  Let $\varphi : \mathbb{D} \longrightarrow \Sigma$ be a wild
  immersion of the disk, \textit{i.e.} that does not extend to $\pD$,
  and let $z \in \pD$ such that the path
  $\gamma : t \mapsto \varphi(tz)$ does not have a limit in $\Sigma$
  when $t \rightarrow 1$. Then $\gamma([0,1) )$ is a hyperbolic closed
  geodesic in $\Sigma$ .
\end{proposition}

\begin{proof}
  According to both Lemma \ref{extension1} and \ref{extension2}
  $\varphi$ extends to a maximal
  $\varphi' : \Hy \longrightarrow \Sigma$ which is invariant by
  multiplication by a certain positive number $\lambda \neq
  1$. $\gamma$ in this extension must head toward $0 \in \partial \Hy$
  for otherwise it would have a limit in $\Sigma$. Its image is
  therefore a closed hyperbolic leave because it is invariant by the
  action of $\lambda$.
\end{proof}

\bibliographystyle{alpha} \bibliography{biblio}
 
\end{document}